\documentclass[smallextended,envcountsect]{svjour3}
\smartqed
\usepackage{graphicx}
\usepackage{xcolor}
\journalname{JOTA}

\usepackage{amssymb}
\usepackage{amsmath}

\usepackage{hyperref}

\def\R{\mathbb{R}}
\def\E{\mathbb{E}}

\def\ba{\begin{array}}
\def\ea{\end{array}}

\newcommand\alert[1]{{#1}}
\newcommand\balert[1]{{#1}}

\begin{document}

\title{Smoothness Parameter of Power of Euclidean Norm
\thanks{Communicated by Liqun Qi.}}


\author{Anton Rodomanov  \and  Yurii Nesterov}

\institute{Anton Rodomanov  \at
             \balert{ICTEAM, Catholic University of Louvain,
             Louvain-la-Neuve, Belgium}\\
              anton.rodomanov@uclouvain.be
           \and
           Yurii Nesterov   \at
              \balert{Center for Operations Research and
              Economics, Catholic University of Louvain,
              Louvain-la-Neuve, Belgium}\\
              yurii.nesterov@uclouvain.be
}

\date{%
  Received: 29 July 2019 / Accepted: 6 March 2020
  / Published online: 27 March 2020 \\
  \textcopyright{} The Author(s) 2020
}

\def\subclassname{{\bfseries Mathematics Subject Classification}\enspace}
\headerboxheight=60pt
\def\makeheadbox{\noindent\small
  Journal of Optimization Theory and Applications (2020) 185:303--326 \\
  \url{https://doi.org/10.1007/s10957-020-01653-6}
}

\maketitle

\begin{abstract}
In this paper, we study derivatives of powers of Euclidean
norm. We prove their H\"older continuity and establish
explicit expressions for the corresponding constants. We
show that these constants are optimal for odd derivatives
and at most two times suboptimal for the even ones. In the
particular case of integer powers, when the H\"older
continuity transforms into the Lipschitz continuity, we
improve this result and obtain the optimal constants.
\end{abstract}



\section{Introduction}

Starting from the paper \cite{nesterov2006cubic}, there has been an increasing interest in
the \emph{cubic regularization} of Newton's method (see, for example,
\cite{nesterov2008accelerating,cartis2011adaptive1,cartis2011adaptive2,%
carmon2016gradient,kohler2017sub,cartis2018global,doikov2018randomized}),
which has some attractive global worst-case complexity
guarantees. The main idea of this method is to approximate
the objective function with its second-order Taylor
approximation, add to it the cube of Euclidean norm with
certain coefficient and then minimize the result to obtain
a new point.

A natural generalization of this approach consists in considering a general high-order Taylor
approximation together with a certain high-order power of Euclidean norm as a regularizer.
This leads to \emph{tensor methods}
\cite{schnabel1991tensor,baes2009estimate,cartis2017improved,gasnikov2019near},
that have recently gained their popularity after it was shown in \cite{nesterov2015implementable}
that one step of the third-order tensor method for minimizing convex functions is comparable with
that of the cubic Newton method.

For some applications, involving functions with H\"older
continuous derivatives, it may also be reasonable to
regularize the models with \emph{fractional} degrees of
the Euclidean norm, as discussed in
\cite{grapiglia2017regularized} and
\cite{grapiglia2019tensor}.

The efficiency of all the aforementioned methods strongly
depends on our possibilities in solving the corresponding
auxiliary problems, that arise at each iteration.
Therefore, it is important to be able to quickly solve
minimization problems regularized by powers of Euclidean
norm.

Two of the most important characteristics of the objective
function, that influence the convergence rate of
minimization algorithms, are the constants of uniform
convexity and H\"older continuity of derivatives. It is
thus important to know these parameters for powers of
Euclidean norm in order to justify the convergence rates
of the related minimization algorithms.

The uniform convexity of powers of Euclidean norm was first investigated
in \cite{vladimirov1978uniformly}, where the authors obtained optimal constants
for all \emph{integer} powers. This result was then generalized to arbitrary \emph{real}
powers in \cite[Lemma~5]{doikov2019minimizing}. Thus, the question of uniform convexity
is completely solved.

\balert{The question of the H\"older continuity of
derivatives of powers of Euclidean norm is more subtle.
There exist only partial results for some special powers.
For example, for any real power between one and two, the
H\"older continuity of the first derivative follows from the
duality between uniform convexity and H\"older smoothness
(see \cite[Lemma~1]{nesterov2015universal}). For any real
power between two and three, the H\"older continuity of the
second derivative has recently been proved in
\cite[Example~2]{doikov2019minimizing}, where some
suboptimal constants have been obtained. However, there are
currently no general results for an arbitrary power.}

Thus, establishing H\"older continuity of derivatives of powers
of Euclidean norm and estimating the corresponding constants
is still an open problem and constitutes the main topic of this work.

This paper is organized as follows. \balert{In
Section~\ref{section:notation}, we introduce notation and
recall important facts on the norm of symmetric
multilinear operators.}

In Section~\ref{section:formula_for_derivatives}, we
derive a general formula for derivatives of powers of
Euclidean norm
(Theorem~\ref{theorem:formula_for_derivative}). The main
object in this formula is a certain family of recursively
defined polynomials
(Definition~\ref{definition:polynomials}). We give the
corresponding definition and provide several examples.

\balert{In Sections~\ref{section:main_properties_of_polynomials}
and \ref{section:holder_constants_of_polynomials}, we
study these polynomials in more detail. We establish
useful identities and prove several important properties
such as symmetry (Proposition~\ref{proposition:parity}),
non-negativity} (Proposition~
\ref{proposition:nonnegativity_of_polynomials})
and monotonicity
(Proposition~\ref{proposition:monotonicity_of_polynomials}).
Section~\ref{section:holder_constants_of_polynomials} is
devoted to estimating the H\"older constants of the
polynomials. The main results in this section are
Theorem~\ref{theorem:extending_holder_continuity_for_polynomials}
and Theorem~\ref{theorem:holder_constant_of_polynomials}.

In Section~\ref{section:holder_contiuity_of_derivatives},
we apply the auxiliary results obtained in the previous
sections for proving H\"older continuity of derivatives of
powers of Euclidean norm. Namely, in \balert{Theorem~
\ref{theorem:lower_bound_for_holder_constant}, we
derive a lower bound for the possible values of H\"older
constants. In Theorem~\ref{theorem:holder_continuity_on_line}, we prove
H\"older continuity of the derivatives along the lines
passing through the origin. Finally, in
Theorem~\ref{theorem:holder_continuity_on_whole_space}, we
extend this result onto the whole space and discuss the
optimality of the constants.

In the final Section~\ref{section:lipschitz_constants_of_derivatives},
we show how to improve our general result for integer
powers, when the H\"older condition corresponds to the
Lipschitz condition.}

\section{Notation and Generalities}
\label{section:notation}

In this text, $\E$ is a finite dimensional real vector
space. Its dual space, composed of all linear functionals
on $\E$, is denoted by $\E^*$. The value of a linear
functional $s \in \E^*$, evaluated at a point $x \in \E$, is
denoted by $\langle s, x \rangle$. To introduce a
Euclidean norm $\| \cdot \|$ on $\E$, we fix a
self-adjoint positive definite operator $B : \E \to \E^*$
and define $\| x \| := \langle B x, x \rangle^{1/2}$.

\balert{For a function $f : G \to \R$, defined on an open set
$G$ in $\E$, and for an integer $p \geq 0$, the $p$-th
derivative of $f$, if exists, is denoted by $D^p f$. This
derivative is a mapping from $G$ to the space of symmetric
$p$-multilinear forms on $\E$.}

Let $L$ be a $p$-multilinear form on $\E$. Its value, evaluated
at $h_1, \dots, h_p \in \E$, is denoted by
$L[h_1, \dots, h_p]$. When $h_1 = \dots = h_p = h$ for
some $h \in \E$, we abbreviate this as $L[h]^p$. The norm
of $L$ is defined in the standard way:
$$
\ba{c}
\| L \| := \max\limits_{ \| h_1 \| = \dots = \| h_p \| = 1}
| L[h_1, \dots, h_p] |.
\ea
$$
If the form $L$ is symmetric, it is known that the maximum
in the above definition can be achieved when all the
vectors are the same:
\begin{equation}\label{eq:norm_of_symmetric_multilinear_operator}
\ba{c}
\| L \| = \max\limits_{ \| h \| = 1} | L[h]^p |
\ea
\end{equation}
(see, for example, Appendix 1 in \cite{nesterov1994interior}).

For $q \in \R$, by $f_q : \E \to \R$ we denote the $q$-th
power of the Euclidean norm:
$$
\ba{c}
f_q(x) := \| x \|^q.
\ea
$$
The main goal of this paper is to establish that, for any integer $p \geq 0$ and
any real $\nu \in [0, 1]$, the $p$-th derivative of $f_{p+\nu}$ is $\nu$-H\"older continuous:
$$
\ba{c}
\| D^p f_{p+\nu}(x_2) - D^p f_{p+\nu}(x_1) \| \leq A_{p, \nu} \| x_2 - x_1 \|
\ea
$$
for all $x_1, x_2 \in \E$, where $A_{p, \nu}$ is an
explicit constant dependent on $p$ and $\nu$.

\section{Derivatives of Powers of Euclidean Norm}
\label{section:formula_for_derivatives}

We start with deriving a general formula for derivatives
of the function $f_q$. The main objects in this formula
\balert{are univariate polynomials, defined below.}

\begin{definition}\label{definition:polynomials}
For each integer $p \geq 0$ and each $q \in \R$, we define a polynomial $g_{p, q} : \R \to \R$
as follows. When $p=0$, we set $g_{p, q}(\tau) := 1$. For all other $p \geq 1$,
$$
\ba{c}
g_{p, q}(\tau) := (1-\tau^2) g_{p-1, q}'(\tau) + (q-p+1) \tau g_{p-1, q}(\tau).
\ea
$$
\end{definition}

\balert{Each polynomial $g_{p, q}$ is a combination of the
previous polynomial $g_{p-1, q}$ and its derivative
$g_{p-1, q}'$. The first five polynomials can be written
explicitly:}
\balert{$$
\ba{rl}
g_{0, q}(\tau) &= 1, \quad g_{1, q}(\tau) \; = q
\tau, \quad g_{2, q}(\tau) \; = q [ (q-2) \tau^2 + 1 ], \\
g_{3, q}(\tau) &= q (q-2) [ (q-4) \tau^3 + 3 \tau ], \\
g_{4, q}(\tau) &= q (q-2) [ (q-4) (q-6) \tau^4 + 6 (q-4) \tau^2 + 3 ].
\ea
$$}

Let us now describe how derivatives of $f_q$ are related to polynomials $g_{p, q}$.

\begin{theorem}\label{theorem:formula_for_derivative}
For any real $q \in \R$, the function $f_q$ is $p$ times differentiable
for all integer $0 \leq p < q$. The corresponding derivatives are
\begin{equation}\label{eq:formula_for_derivative}
\ba{c}
D^p f_q(x)[h]^p = \| x \|^{q-p} g_{p, q}(\tau_h(x)),
\ea
\end{equation}
where $h \in \E$ is an arbitrary unit vector and
\begin{equation}\label{eq:definition_tau_h_x}
\ba{c}
\tau_h(x) :=
\begin{cases}
\frac{\langle B x, h \rangle}{\| x \|}, & \text{if $x \neq 0$}, \\
0, & \text{if $x = 0$}.
\end{cases}
\ea
\end{equation}
\end{theorem}

\begin{proof}
Note that $f_q$ is infinitely differentiable on $\E
\setminus \{0\}$ since its restriction on this set is a
composition of two infinitely differentiable functions,
namely the quadratic function $\E \setminus \{0\} \to \R:
x \mapsto \| x \|^2 = \langle B x, x \rangle$ and the
power function $]0, +\infty[ \to \R: t \mapsto t^{q/2}$.
Hence, we only need to prove that $f_q$ is also $p$ times
differentiable at the origin for any $0 \leq p < q$, and
that \eqref{eq:formula_for_derivative} holds.

\balert{We proceed by induction. The case $p=0$ is trivial
since, by definition, the zeroth derivative of a function is the function
itself, while $g_{0, q}(\tau) = 1$ for any $\tau
\in \R$. Let us assume that $p \geq 1$, and the claim is
proved for $p' := p-1$.}

First, let us justify \eqref{eq:formula_for_derivative}
for any $x \in \E \setminus \{0\}$. By the induction
hypothesis,
$$
\ba{c}
D^{p-1} f_q(x)[h]^{p-1} = \| x \|^{q-p+1} g_{p-1, q}(\tau_h(x))
\ea
$$
for all $x \in \E$. Differentiating, we obtain that
$$
\ba{c}
D \| \cdot \|(x)[h] = \tau_h(x), \qquad
D \tau_h(x)[h] = \frac{1 - \tau_h^2(x)}{ \| x \| }
\ea
$$
for all $x \in \E \setminus \{0\}$, and hence
$$
\ba{rl}
D^p f_q(x)[h]^p &=
\| x \|^{q-p+1} g_{p-1, q}'(\tau_h(x)) \frac{1 - \tau_h^2(x)}{ \| x \| } \\
&\qquad + (q-p+1) \| x \|^{q-p} \tau_h(x) g_{p-1, q}(\tau_h(x)) \\
&= \| x \|^{q-p} [ (1 - \tau_h^2(x)) g_{p-1, q}'(\tau_h(x))
+ \tau_h(x) g_{p-1, q}(\tau_h(x)) ] \\
&= \| x \|^{q-p} g_{p, q}(\tau_h(x)),
\ea
$$
where the last equality follows from Definition~\ref{definition:polynomials}.

Now let us show that $f_q$ is also $p$ times differentiable at the origin
with $D^p f_q(0) = 0$ (this is what \eqref{eq:formula_for_derivative} says when $x=0$).
By our inductive assumption, we already know that $D^{p-1} f_q(0) = 0$.
\balert{Therefore, according to the definition of
derivative, it remains to show that $\lim_{x \to 0; x \neq
0} \frac{ \| D^{p-1} f_q(x) \| }{\| x \| } = 0$}, or,
equivalently, in view of 
\eqref{eq:norm_of_symmetric_multilinear_operator}, that
$$
\ba{c}
\lim\limits_{x \to 0; x \neq 0} \max\limits_{\| h \| = 1}
\frac{ | D^{p-1} f_q(x)[h]^{p-1} | }{ \| x \| } = 0.
\ea
$$
Applying our inductive assumption, we obtain that
\begin{equation}\label{eq:formula_for_derivative_proof1}
\ba{c}
\max\limits_{\| h \| = 1} \frac{ | D^{p-1} f_q(x)[h]^{p-1} |
}{\| x \| } = \| x \|^{q-p} \max\limits_{\| h \| = 1} |g_
{p-1, q}(\tau_h(x))|
\ea
\end{equation}
\balert{for all $x \in \E \setminus \{0\}$. Since $p < q$, we
have $\| x \|^{q-p} \to 0$ as $x \to 0$. Thus, we need to
show that $|g_{p-1, q}(\tau_h(x))|$ is uniformly bounded for
all $x \in \E$ and all unit $h \in
\E$. Indeed, by Cauchy-Schwartz inequality, we have
$|\tau_h(x)| \leq 1$. Hence $|g_{p-1, q}(\tau_h(x))| \leq
\max_{[-1, 1]} |g_{p-1, q}|$.} The right-hand side in the
above inequality is finite, since a continuous function
always achieves its maximum on a compact interval.\qed
\end{proof}

\section{Main Properties of Polynomials}
\label{section:main_properties_of_polynomials}

Let us study the polynomials $g_{p, q}$ introduced in
Definition~\ref{definition:polynomials}. Our first
observation is that $g_{p, q}$, as a function, is always either even or odd.
\begin{proposition}\label{proposition:parity}
For any integer $p \geq 0$, and any $q \in \R$, $g_{p, q}$ has the same parity as $p$,
i.e. $g_{p, q}(-\tau) = (-1)^p g_{p, q}(\tau)$ for all $\tau \in \R$.
\end{proposition}

\begin{proof}
Easily follows from Definition~\ref{definition:polynomials} by induction.\qed
\end{proof}

\balert{Next we establish identities with the first and
second derivatives of $g_{p, q}$.}

\begin{lemma}\label{lemma:formula_for_derivative_of_polynomial_1}
For any integer $p \geq 1$, and any $q, \tau \in \R$,
\begin{equation}\label{eq:formula_for_derivative_of_polynomial_1}
g_{p, q}'(\tau) = (1-\tau^2) g_{p-1, q}''(\tau) +
(q-p-1) \tau g_{p-1, q}'(\tau) + (q-p+1) g_{p-1, q}(\tau).
\end{equation}
\end{lemma}

\begin{proof}
Follows from Definition~\ref{definition:polynomials} using standard
rules of differentiation.\qed
\end{proof}

\begin{lemma}\label{lemma:formula_for_qmpg}
For any integer $p \geq 0$, and any $q, \tau \in \R$,
$$
\ba{c}
(q-p) g_{p, q}(\tau) = \tau g_{p, q}'(\tau) + q g_{p, q-2}(\tau).
\ea
$$
\end{lemma}

\begin{proof}
We proceed by induction on $p$. For $p=0$, by Definition~\ref{definition:polynomials},
we have $(q-p) g_{p, q}(\tau) = q$ while $\tau g_{p, q}'(\tau) = 0$ and
$q g_{p, q-2}(\tau) = q$, so the claim is obviously true. Now let us prove the claim
for $p \geq 1$, assuming that it is already true for all integer $0 \leq p' \leq p-1$.
By Definition~\ref{definition:polynomials}, we have
$$
\ba{c}
(q-p) g_{p, q}(\tau) = (q-p) ( (1-\tau^2) g_{p-1, q}'(\tau) + (q-p+1) \tau g_{p-1, q}(\tau) ).
\ea
$$
Rearranging, we obtain
$$
\ba{rl}
(q-p) g_{p, q}(\tau) = & (q-p-1) \tau (q-p+1) g_{p-1, q}(\tau) \\
& + (1-\tau^2) (q-p) g_{p-1, q}'(\tau) + (q-p+1) \tau
g_{p-1, q}(\tau).
\ea
$$
By the induction hypothesis, applied for $p' := p-1$, we have
$$
\ba{c}
(q-p+1) g_{p-1, q}(\tau) = \tau g_{p-1, q}'(\tau) + q g_{p-1, q-2}(\tau).
\ea
$$
for all $\tau \in \R$. Differentiating both sides, we obtain from this that
$$
\ba{c}
(q-p) g_{p-1, q}'(\tau) = \tau g_{p-1, q}''(\tau) + q g_{p-1, q-2}'(\tau).
\ea
$$
Combining the above three formulas, we see that
\begin{equation}\label{eq:formula_for_qmpg:proof1}
\ba{rl}
(q-p) g_{p, q}(\tau) = & (q-p-1) \tau ( \tau g_{p-1, q}' + q g_{p-1, q-2}(\tau) ) \\
& + (1-\tau^2) ( \tau g_{p-1, q}''(\tau) + q g_{p-1, q-2}'(\tau) ) \\
& + (q-p+1) \tau g_{p-1, q}(\tau).
\ea
\end{equation}

At the same time, by Lemma~\ref{lemma:formula_for_derivative_of_polynomial_1}, we have
$$\balert{
\tau g_{p, q}'(\tau) = (1-\tau^2) \tau g_{p-1, q}''(\tau)
+ (q-p-1) \tau^2 g_{p-1, q}'(\tau) + (q-p+1) \tau g_{p-1, q}(\tau),
}$$
and, by Definition~\ref{definition:polynomials}, we also have
$$
\ba{c}
q g_{p, q-2}(\tau) = (1-\tau^2) q g_{p-1, q-2}'(\tau) + (q-p-1) \tau q g_{p-1, q-2}(\tau).
\ea
$$
Summing the above two identities, we obtain the right-hand side of
\eqref{eq:formula_for_qmpg:proof1}.\qed
\end{proof}

\begin{lemma}\label{lemma:formula_for_derivative_of_polynomial_2}
For any integer $p \geq 1$, and any $q, \tau \in \R$,
$$
\ba{c}
g_{p, q}'(\tau) = (1-\tau^2) g_{p-1, q}''(\tau)
+ (q-p) \tau g_{p-1, q}'(\tau) + q g_{p-1, q-2}(\tau).
\ea
$$
\end{lemma}

\begin{proof}
Apply Lemma~\ref{lemma:formula_for_qmpg} to the last term in
\eqref{eq:formula_for_derivative_of_polynomial_1}.\qed
\end{proof}

The following lemma is particularly interesting. It turns out that,
up to a constant factor, the derivative of the polynomial $g_{p, q}$
is exactly the previous polynomial but
with a shifted value of $q$.

\begin{lemma}\label{lemma:formula_for_derivative_of_polynomial_3}
For any integer $p \geq 1$, and any $q \in \R$, we have $g_{p, q}' = p q g_{p-1, q-2}$.
\end{lemma}

\begin{proof}
We proceed by induction on $p$. Let $\tau \in \R$. For
$p=1$, we know from
Definition~\ref{definition:polynomials} that $g_{p,
q}(\tau) = q \tau$, while $p q g_{p-1, q-2}(\tau) = q$,
therefore the claim is indeed true. Now let us prove the
claim for $p \geq 2$ assuming that it is already proved
for all integer $0 \leq p' \leq p-1$. From
Lemma~\ref{lemma:formula_for_derivative_of_polynomial_2},
we already know that
$$
\ba{c}
g_{p, q}'(\tau) = (1-\tau^2) g_{p-1, q}''(\tau)
+ (q-p) \tau g_{p-1, q}'(\tau)
+ q g_{p-1, q-2}(\tau).
\ea
$$
Therefore it remains to prove that
$$
\ba{c}
(1-\tau^2) g_{p-1, q}''(\tau) + (q-p) \tau g_{p-1, q}'(\tau) = (p-1) q g_{p-1, q-2}(\tau).
\ea
$$
\balert{By the induction hypothesis for $p' := p-1$, we
already have the identity $g_{p-1, q}' = (p-1) q g_{p-2,
q-2}$ and in particular $g_{p-1, q}'' = (p-1) q \alert{g_
{p-2, q-2}'}$. Thus,}
$$
\ba{rl}
(1-\tau^2) g_{p-1, q}''&(\tau) + (q-p) \tau g_{p-1, q}'(\tau) \\
&= (p-1) q [ (1-\tau^2) g_{p-2, q-2}'(\tau) + (q-p) \tau g_{p-2, q-2}(\tau) ].
\ea
$$
It remains to verify that
$$
\ba{c}
(1-\tau^2) g_{p-2, q-2}'(\tau) + (q-p) \tau g_{p-2, q-2}(\tau) = g_{p-1, q-2}(\tau).
\ea
$$
But this is given directly by Definition~\ref{definition:polynomials}.\qed
\end{proof}

Combined with Definition~\ref{definition:polynomials},
Lemma~\ref{lemma:formula_for_derivative_of_polynomial_3}
gives us a useful recursive formula for $g_{p, q}$, that
does not involve any derivatives.

\begin{lemma}\label{lemma:formula_for_polynomials_without_derivatives}
For any integer $p \geq 2$, and any $q, \tau \in \R$,
\begin{equation}\label{eq:formula_for_polynomials_without_derivatives}
\ba{c}
g_{p, q}(\tau) = (1-\tau^2) (p-1) q g_{p-2, q-2}(\tau) + (q-p+1) \tau g_{p-1, q}(\tau).
\ea
\end{equation}
\end{lemma}

Lemma~\ref{lemma:formula_for_polynomials_without_derivatives}
has several corollaries. The first one gives us
closed-form expressions for the values of $g_{p, q}$ at
the boundary points of the interval $[0, 1]$.

\begin{proposition}\label{proposition:boundary_values_of_polynomials}
For any integer $p \geq 0$, and any $q \in \R$, \alert{we
have}\footnote{\alert{For a positive integer $n$, by $n!!$
we denote the double factorial of $n$ (the product of all
integers between 1 and $n$ with the same parity as $n$).
We also define $(-1)!! = 0!! = 1$.}}
\begin{equation}\label{eq:boundary_value_of_polynomial_0}
\ba{c}
g_{p, q}(0) =
\begin{cases}
(p-1)!! \prod_{i=0}^{\frac{p}{2}-1} (q - 2i), & \text{if $p$ even}, \\
0, & \text{if $p$ odd},
\end{cases}
\ea
\end{equation}
and
\begin{equation}\label{eq:boundary_value_of_polynomial_1}
\ba{c}
g_{p, q}(1) = \prod_{i=0}^{p-1} (q-i).
\ea
\end{equation}
\end{proposition}

\begin{proof}
\alert{We proceed by induction on $p$. From
Definition~\ref{definition:polynomials}, we have $g_{0,
q}(0) = g_{0, q}(1) = 1$ and $g_{1, q}(0) = 0$, $g_{1, q}(1)
= q$. Thus, the claim is indeed true for $p=0$ and
$p=1$. Now let us prove the claim for $p \geq 2$ assuming
that it is already true for all integer $0 \leq p' \leq
p-1$. Using
Lemma~\ref{lemma:formula_for_polynomials_without_derivatives},
we obtain}
\begin{equation}\label{eq:boundary_values_of_polynomials_proof1}
\ba{c}
g_{p, q}(0) = (p-1) q g_{p-2, q-2}(0).
\ea
\end{equation}
By the induction hypothesis, applied for $p' := p-2$ (and $q' := q-2$), we have
$$
\ba{c}
g_{p-2, q-2}(0) =
\begin{cases}
(p-3)!! \alert{\prod_{i=0}^{\frac{p}{2}-2}} (q-2-2i), & 
\text{if $p$ is even}, \\
0, & \text{if $p$ is odd}.
\end{cases}
\ea
$$
By shifting the index in the product, this can be rewritten as
$$
\ba{c}
g_{p-2, q-2}(0) =
\begin{cases}
(p-3)!! \alert{\prod_{i=1}^{\frac{p}{2}-1}} (q-2i), & 
\text{if $p$ is even}, \\
0, & \text{if $p$ is odd}.
\end{cases}
\ea
$$
Substituting this into \eqref{eq:boundary_values_of_polynomials_proof1},
we obtain \eqref{eq:boundary_value_of_polynomial_0}.

\balert{Similarly, by Lemma~\ref{lemma:formula_for_polynomials_without_derivatives}, we
also have $g_{p, q}(1) = (q-p+1) g_{p-1, q}(1)$. But by the
induction hypothesis, $g_{p-1, q}(1) = \prod_{i=0}^{p-2} 
(q-i)$, and we obtain
\eqref{eq:boundary_value_of_polynomial_1}.}\qed
\end{proof}

The second corollary of Lemma~\ref{lemma:formula_for_polynomials_without_derivatives}
states that $g_{p, q}$ cannot take negative values on the interval $[0, 1]$,
provided that $q$ is sufficiently large.

\begin{proposition}\label{proposition:nonnegativity_of_polynomials}
For any integer $p \geq 0$, and any real $q \geq p-1$, $g_{p, q}$ is non-negative on $[0, 1]$.
\end{proposition}

\begin{proof}
We proceed by induction on $p$. Let $0 \leq \tau \leq 1$. For $p=0$, we know,
by Definition~\ref{definition:polynomials}, that $g_{p, q}(\tau) = 1$, which is
actually non-negative for all real $q$.
For $p=1$, by Definition~\ref{definition:polynomials}, we have $g_{p, q}(\tau) = q \tau$,
which is indeed non-negative when $q \geq p-1 = 0$.

Now let us prove the claim for $p \geq 2$, assuming that it
is already proved for all integer $0 \leq p' \leq p-1$.
From
Lemma~\ref{lemma:formula_for_polynomials_without_derivatives},
we know that
$$
\ba{c}
g_{p, q}(\tau) = (1-\tau^2) (p-1) q g_{p-2, q-2}(\tau) + (q-p+1) \tau g_{p-1, q}(\tau).
\ea
$$
By the induction hypothesis, applied respectively for $p' := p-2$, $q' := q-2$ and
$p' := p-1$, $q' := q$ (observe that in both cases $q' \geq p'-1$ since $q \geq p$),
we have $g_{p-2, q-2}(\tau) \geq 0$ and $g_{p-1, q}(\tau) \geq 0$.
Since $q \geq p-1 \geq 1$, then also $q-p+1 \geq 0$, and $(p-1) q \geq 0$.
Thus, all parts in the right-hand side of the above formula are non-negative.\qed
\end{proof}

Combining
Proposition~\ref{proposition:nonnegativity_of_polynomials}
with
Lemma~\ref{lemma:formula_for_derivative_of_polynomial_3},
we obtain that, when $q \geq p$, the polynomial $g_{p, q}$
is not only non-negative but also monotonically
increasing.

\begin{proposition}\label{proposition:monotonicity_of_polynomials}
For any integer $p \geq 0$, and any real $q \geq p$, the
derivative $g_{p, q}'$ is non-negative on $[0, 1]$;
hence $g_{p, q}$ is monotonically increasing on $[0, 1]$.
\end{proposition}

Finally, let us show how we can apply the properties, that we have established above,
to finding the maximal absolute value of $g_{p, q}$ on $[-1, 1]$.

\begin{proposition}\label{proposition:maximal_absolute_value_of_polynomial}
For any integer $p \geq 0$, and any real $q \geq p$,
$$
\ba{c}
\max_{[-1, 1]} |g_{p, q}| = \prod_{i=0}^{p-1} (q-i).
\ea
$$
\end{proposition}

\begin{proof}
By Proposition~\ref{proposition:parity}, we have
$\max_{[-1, 1]} | g_{p, q} | = \max_{[0, 1]} |g_{p, q}|$.
Since $g_{p, q}$ is non-negative on $[0, 1]$
(Proposition~\ref{proposition:nonnegativity_of_polynomials}),
$\max_{[0, 1]} |g_{p, q}| = \max_{[0, 1]} g_{p, q}$.
By Proposition~\ref{proposition:monotonicity_of_polynomials},
$\max_{[0, 1]} g_{p, q} = g_{p, q}(1)$. But
$g_{p, q}(1) = \prod_{i=0}^{p-1} (q-i)$ according to
Proposition~\ref{proposition:boundary_values_of_polynomials}.\qed
\end{proof}

\section{H\"older Constants of Polynomials}
\label{section:holder_constants_of_polynomials}

We continue our study of polynomials $g_{p, q}$, but now we restrict our attention
to the particular case when $q=p+\nu$ for some real $\nu \in [0, 1]$.

Clearly, the polynomial $g_{p, p+\nu}$ is $\nu$-H\"older
continuous on $[-1, 1]$, since this is true for any other
polynomial on a compact interval. The goal of this section
is to obtain an explicit expression for the corresponding
\emph{H\"older constant}. \balert{We start with the result,
allowing us to reduce our task to that on $[0, 1]$.}

\begin{theorem}\label{theorem:extending_holder_continuity_for_polynomials}
For any integer $p \geq 0$, and any real $\nu \in [0, 1]$, the polynomial
$g_{p, p+\nu}$ is $\nu$-H\"older continuous on $[-1, 1]$ with constant
\alert{$$
\ba{c}
\tilde{H}_{p, \nu} :=
\begin{cases}
H_{p, \nu}, & \text{if $p$ is even}, \\
2^{1-\nu} H_{p, \nu}, & \text{if $p$ is odd},
\end{cases}
\ea
$$}
where $H_{p, \nu}$ is the corresponding H\"older constant of $g_{p, p+\nu}$ on $[0, 1]$.
\end{theorem}

\begin{proof}
Let $\tau_1, \tau_2 \in [-1, 1]$. We need to prove that
\begin{equation}\label{eq:extending_holder_continuity_for_polynomials_proof1}
\ba{c}
| g_{p, p+\nu}(\tau_2) - g_{p, p+\nu}(\tau_1) |
\leq \tilde{H}_{p, \nu} | \tau_2 - \tau_1 |^\nu.
\ea
\end{equation}

By Proposition~\ref{proposition:parity}, this inequality is invariant to negation transformations
$(\tau_1, \tau_2) \mapsto (-\tau_1, -\tau_2)$. Therefore, we can assume that $\tau_2 \geq 0$.
Furthermore, we can assume that $\tau_1 < 0$, since otherwise the claim is trivial.

\balert{\underline{Case I}. Suppose $p$ is even. Then, by
Proposition~\ref{proposition:parity},}
$$
\ba{c}
| g_{p, p+\nu}(\tau_2) - g_{p, p+\nu}(\tau_1) | = | g_{p, p+\nu}(\tau_2) - g_{p, p+\nu}(-\tau_1) |.
\ea
$$
Note that $-\tau_1, \tau_2 \in [0, 1]$. Therefore, by H\"older condition on $[0, 1]$,
$$
\ba{c}
| g_{p, p+\nu}(\tau_2) - g_{p, p+\nu}(-\tau_1) | \leq H_{p, \nu} |\tau_2 + \tau_1|^\nu.
\ea
$$
At the same time, $|\tau_2 + \tau_1| \leq \tau_2 - \tau_1$ by the triangle inequality,
and \eqref{eq:extending_holder_continuity_for_polynomials_proof1} follows.

\underline{Case II}. Now suppose $p$ is odd. By Proposition~\ref{proposition:parity}
and Proposition~\ref{proposition:nonnegativity_of_polynomials},
$$
\ba{c}
| g_{p, p+\nu}(\tau_2) - g_{p, p+\nu}(\tau_1) |
= g_{p, p+\nu}(\tau_2) + g_{p, p+\nu}(-\tau_1).
\ea
$$
Recall that $g_{p, p+\nu}(0) = 0$ (Proposition~\ref{proposition:boundary_values_of_polynomials}).
Therefore,
$$
\ba{rl}
g_{p, p+\nu}(\tau_2) &= g_{p, p+\nu}(\tau_2) - g_{p, p+\nu}
(0) \leq H_{p, \nu} \tau_2^\nu, \\
g_{p, p+\nu}(-\tau_1) &= g_{p, p+\nu}(-\tau_1) - g_{p, p+\nu}(0) \leq H_{p, \nu} (-\tau_1)^\nu.
\ea
$$
Hence,
$$
\ba{c}
g_{p, p+\nu}(\tau_2) + g_{p, p+\nu}(-\tau_1) \leq H_{p, \nu} ( \tau_2^\nu + (-\tau_1)^\nu ).
\ea
$$
\balert{To prove
\eqref{eq:extending_holder_continuity_for_polynomials_proof1},
it remains to show that $\tau_2^\nu + (-\tau_1)^\nu \leq 2^
{1-\nu} (\tau_2 - \tau_1)^\nu$. But this follows from the
concavity of power function $t \mapsto t^\nu$.}\qed
\end{proof}

Our next task is to estimate the H\"older constant of
$g_{p, p+\nu}$ on $[0, 1]$:
\begin{equation}\label{eq:def_H_p_nu}
\ba{c}
H_{p, \nu} := \max\limits_{0 \leq \tau_1 < \tau_2 \leq 1}
\frac{ g_{p, p+\nu}(\tau_2) - g_{p, p+\nu}(\tau_1) }{ (\tau_2 - \tau_1)^\nu }.
\ea
\end{equation}
\balert{Note that
Proposition~\ref{proposition:monotonicity_of_polynomials}
allows us to remove the absolute value sign.}

\begin{theorem}\label{theorem:holder_constant_of_polynomials}
For any integer $p \geq 0$, and any real $\nu \in [0, 1]$, we have
\begin{equation}\label{eq:holder_constant_of_polynomials}
\ba{c}
H_{p, \nu} \leq \prod_{i=1}^p (\nu+i).
\ea
\end{equation}
\end{theorem}

The proof of Theorem~\ref{theorem:holder_constant_of_polynomials}
is based on two auxiliary propositions.

\begin{proposition}\label{proposition:monotone_fraction}
For any integer $p \geq 0$ and any real $\nu, \tau_1 \in [0, 1]$, the function
\begin{equation}\label{eq:proposition_monotone_fraction}
\ba{c}
]\tau_1, +\infty[ \to \R:
\tau_2 \mapsto \frac{g_{p, p+\nu}(\tau_2) - g_{p, p+\nu}(\tau_1)}{(\tau_2 - \tau_1)^\nu}
\ea
\end{equation}
is monotonically increasing on $]\tau_1, 1]$.
\end{proposition}

\begin{proposition}\label{proposition:monotone_fraction1}
For any integer $p \geq 0$ and any real $\nu \in [0, 1]$, the function
\begin{equation}\label{eq:proposition:monotone_fraction1}
\ba{c}
]0, 1] \to \R: \tau \mapsto \frac{g_{p, p+\nu}(\tau)}{1 - (1-\tau)^\nu}
\ea
\end{equation}
is monotonically decreasing on $]0, 1]$.
\end{proposition}

Let us assume for a moment that these propositions are
already proved. Then, the proof of
Theorem~\ref{theorem:holder_constant_of_polynomials} is
simple.

\begin{proof}
Let $0 \leq \tau_1 < \tau_2 \leq 1$. From Proposition~\ref{proposition:monotone_fraction},
we know that
$$
\ba{c}
\frac{ g_{p, p+\nu}(\tau_2) - g_{p, p+\nu}(\tau_1) }{ (\tau_2 - \tau_1)^\nu }
\leq \frac{ g_{p, p+\nu}(1) - g_{p, p+\nu}(\tau_1) }{ (1-\tau_1)^\nu }.
\ea
$$
Therefore, to prove \eqref{eq:holder_constant_of_polynomials}, it remains to show that
$$
\ba{c}
\frac{ g_{p, p+\nu}(1) - g_{p, p+\nu}(\tau_1) }{ (1-\tau_1)^\nu }
\leq \prod_{i=1}^p (\nu+i).
\ea
$$

\balert{Recall that, by Proposition~\ref{proposition:boundary_values_of_polynomials}, we have
$\prod_{i=1}^p (\nu+i) = g_{p, p+\nu}(1)$.} Thus, the
inequality we need to prove is
$$
\ba{c}
\frac{ g_{p, p+\nu}(1) - g_{p, p+\nu}(\tau_1) }{ (1-\tau_1)^\nu } \leq g_{p, p+\nu}(1),
\ea
$$
or, equivalently,
$$
\ba{c}
\frac{ g_{p, p+\nu}(\tau_1) }{ 1 - (1-\tau_1)^\nu } \geq g_{p, p+\nu}(1).
\ea
$$
But this follows from Proposition~\ref{proposition:monotone_fraction1}.
\end{proof}

Our goal now is to prove Proposition~\ref{proposition:monotone_fraction}
and Proposition~\ref{proposition:monotone_fraction1}.

We start with Proposition~\ref{proposition:monotone_fraction}. It requires three
technical lemmas.

\begin{lemma}\label{lemma:inequality_g_minus_tau_derivative}
For any integer $p \geq 0$, and any real $\nu, \tau \in [0, 1]$,
\begin{equation}\label{eq:inequality_g_minus_tau_derivative_1}
\ba{c}
g_{p, p+\nu}(\tau) \geq \tau g_{p, p+\nu}'(\tau).
\ea
\end{equation}
Moreover, when $p \geq 2$,
\begin{equation}\label{eq:lemma_inequality_g_minus_tau_derivative}
\ba{rl}
g_{p, p+\nu}(\tau) & - \tau g_{p, p+\nu}'(\tau) \\
&\geq (1-\tau^2) (p-1) (p+\nu) (g_{p-2, p-2+\nu}(\tau) - \tau g_{p-2, p-2+\nu}'(\tau)).
\ea
\end{equation}
\end{lemma}

\begin{proof}
First, let us prove \eqref{eq:lemma_inequality_g_minus_tau_derivative}.
By Lemma~\ref{lemma:formula_for_derivative_of_polynomial_1}, we have
$$
\ba{c}
g_{p, p+\nu}'(\tau) = (1-\tau^2) g_{p-1, p+\nu}''(\tau)
+ (\nu-1) \tau g_{p-1, p+\nu}'(\tau) + (\nu+1) g_{p-1, p+\nu}(\tau).
\ea
$$
Since $g_{p-1, p+\nu}'(\tau) \geq 0$
(Proposition~\ref{proposition:monotonicity_of_polynomials})
and $\nu \leq 1$, it follows that
$$
\ba{c}
g_{p, p+\nu}'(\tau) \leq (1-\tau^2) g_{p-1, p+\nu}''(\tau) + (\nu+1) g_{p-1, p+\nu}(\tau).
\ea
$$
At the same time, by Definition~\ref{definition:polynomials},
$$
\ba{c}
g_{p, p+\nu}(\tau) = (1-\tau^2) g_{p-1, p+\nu}'(\tau) + 
\alert{(\nu+1) \tau g_{p-1, p+\nu}(\tau)}.
\ea
$$
Thus,
$$
\ba{c}
g_{p, p+\nu}(\tau) - \tau g_{p, p+\nu}'(\tau)
\geq (1-\tau^2) (g_{p-1, p+\nu}'(\tau) - \tau g_{p-1, p+\nu}''(\tau)).
\ea
$$
Applying Lemma~\ref{lemma:formula_for_derivative_of_polynomial_3}, we obtain that
$$
\ba{rl}
g_{p-1, p+\nu}'(\tau) &= (p-1) (p+\nu) g_{p-2, p-2+\nu}(\tau), \\
g_{p-1, p+\nu}''(\tau) &= (p-1) (p+\nu) g_{p-2, p-2+\nu}'(\tau),
\ea
$$
and \eqref{eq:lemma_inequality_g_minus_tau_derivative} follows.

It remains to prove \eqref{eq:inequality_g_minus_tau_derivative_1}.
For $p=0$, we have $g_{p, p+\nu}(\tau) = 1$ (Definition~\ref{definition:polynomials}),
hence $\tau g_{p, p+\nu}'(\tau) = 0$, and \eqref{eq:inequality_g_minus_tau_derivative_1}
is indeed true. For $p=1$, by Definition~\ref{definition:polynomials}, we have
$g_{p, p+\nu}(\tau) = (p+\nu) \tau$,
hence $\tau g_{p, p+\nu}'(\tau) = (p+\nu) \tau$,
and \eqref{eq:inequality_g_minus_tau_derivative_1} is again true.
The general case $p \geq 2$ easily follows
from \eqref{eq:lemma_inequality_g_minus_tau_derivative} by induction.\qed
\end{proof}

\begin{lemma}\label{lemma:inequality_qgpm2}
For any integer $p \geq 0$, any real $\nu \in [0, 1]$, and
$0 \leq \tau_1 \leq \tau_2 \leq 1$,
\begin{equation}\label{eq:lemma_inequality_qgpm2}
\ba{c}
(p+\nu) g_{p, p-2+\nu}(\tau_2) \leq \nu (g_{p, p+\nu}(\tau_1) - \tau_1 g_{p, p+\nu}'(\tau_1)).
\ea
\end{equation}
\end{lemma}

\begin{proof}
We use induction in $p$. For $p=0$, we have $g_{p,
p-2+\nu}(\tau_2) = 1$, while $g_{p, p+\nu}(\tau_1) - \tau_1
g_{p, p+\nu}'(\tau_1) = 1$ (see
Definition~\ref{definition:polynomials}), so the claim is
true. For $p=1$, we have $g_{p, p-2+\nu}(\tau_2) =
-(1-\nu) \tau_2 \leq 0$ while $g_{p, p+\nu}(\tau_1) -
\tau_1 g_{p, p+\nu}'(\tau_1) = 0$, (see
Definition~\ref{definition:polynomials}), hence the claim
is again true.

Now we prove the claim for $p \geq 2$ assuming that it is already true
for all integer $0 \leq p' \leq p-1$.
According to Lemma~\ref{lemma:formula_for_polynomials_without_derivatives}, we have
$$
\ba{rl}
g_{p, p-2+\nu}(\tau_2) =
&(1-\tau_2^2) (p-1) (p-2+\nu) g_{p-2, p-4+\nu}(\tau_2) \\
&- (1-\nu) \tau_2 g_{p-1, p-2+\nu}(\tau_2).
\ea
$$
Since $g_{p-1, p-2+\nu}(\tau_2) \geq 0$
(Proposition~\ref{proposition:nonnegativity_of_polynomials}), we further have
$$
\ba{c}
g_{p, p-2+\nu}(\tau_2) \leq (1-\tau_2^2) (p-1) (p-2+\nu) g_{p-2, p-4+\nu}(\tau_2).
\ea
$$

If $g_{p-2, p-4+\nu}(\tau_2) \leq 0$,
it follows that $g_{p, p-2+\nu}(\tau_2) \leq 0$, and the proof in this case
is finished, because the right-hand side in \eqref{eq:lemma_inequality_qgpm2} is
always non-negative in view of Lemma~\ref{lemma:inequality_g_minus_tau_derivative}.
Therefore, we can assume that $g_{p-2, p-4+\nu}(\tau_2) \geq 0$.

Since $\tau_2 \geq \tau_1$, then
$$
\ba{c}
g_{p, p-2+\nu}(\tau_2) \leq (1-\tau_1^2) (p-1) (p-2+\nu) g_{p-2, p-4+\nu}(\tau_2).
\ea
$$
Applying the inductive assumption to $p' := p-2$, we obtain
$$
\ba{c}
(p-2+\nu) g_{p-2, p-4+\nu}(\tau_2)
\leq \nu (g_{p-2, p-2+\nu}(\tau_1) - \tau_1 g_{p-2, p-2+\nu}'(\tau_1)).
\ea
$$
Hence,
$$
\ba{c}
g_{p, p-2+\nu}(\tau_2) \leq \nu (1-\tau_1^2) (p-1) (g_{p-2, p-2+\nu}(\tau_1)
- \tau_1 g_{p-2, p-2+\nu}'(\tau_1)).
\ea
$$
Thus, to finish the proof, it remains to show that
$$
\ba{rl}
(1-\tau_1^2) (p-1) (p+\nu) (g_{p-2, p-2+\nu}&(\tau_1) - \tau_1 g_{p-2, p-2+\nu}'(\tau_1)) \\
& \leq g_{p, p+\nu}(\tau_1) - \tau_1 g_{p, p+\nu}'(\tau_1).
\ea
$$
But this is guaranteed by Lemma~\ref{lemma:inequality_g_minus_tau_derivative}.\qed
\end{proof}

\begin{lemma}\label{lemma:auxiliary_function_for_monotone_fraction}
For any integer $p \geq 0$, and any real $\nu, \tau_2 \in [0, 1]$, the function
\begin{equation}\label{eq:lemma_auxiliary_function_for_monotone_fraction}
\ba{c}
]0, +\infty[ \to \R:
\tau_1 \mapsto \frac{ \nu g_{p, p+\nu}(\tau_1) - (p+\nu) g_{p, p-2+\nu}(\tau_2) }{\tau_1}
\ea
\end{equation}
is monotonically decreasing on $]0, \tau_2]$.
\end{lemma}

\begin{proof}
The function \eqref{eq:lemma_auxiliary_function_for_monotone_fraction} is differentiable
with derivative
$$
\ba{c}
\frac{ \nu (\tau_1 g_{p, p+\nu}'(\tau_1) - g_{p, p+\nu}(\tau_1))
+ (p+\nu) g_{p, p-2+\nu}(\tau_2) }{ \tau_1^2 },
\ea
$$
which is non-positive on $]0, \tau_2]$ by Lemma~\ref{lemma:inequality_qgpm2}.\qed
\end{proof}

Now we can present the proof of Proposition~\ref{proposition:monotone_fraction}:

\begin{proof}
Since \eqref{eq:proposition_monotone_fraction} is differentiable,
it suffices to prove that its derivative
$$
\ba{c}
\frac{g_{p, p+\nu}'(\tau_2)}{(\tau_2 - \tau_1)^\nu}
- \frac{\nu (g_{p, p+\nu}(\tau_2) - g_{p, p+\nu}(\tau_1))}{(\tau_2 - \tau_1)^{\nu+1}}
\ea
$$
is non-negative for all $0 < \tau_1 < \tau_2 \leq 1$, or, equivalently, that
$$
\ba{c}
g_{p, p+\nu}'(\tau_2) (\tau_2 - \tau_1) \geq \nu (g_{p, p+\nu}(\tau_2) - g_{p, p+\nu}(\tau_1)).
\ea
$$
By Lemma~\ref{lemma:formula_for_qmpg},
\begin{equation}\label{eq:proposition:monotone_fraction:proof:eq1}
\ba{c}
\nu g_{p, p+\nu}(\tau_2) = \tau_2 g_{p, p+\nu}'(\tau_2) + (p+\nu) g_{p, p-2+\nu}(\tau_2).
\ea
\end{equation}
Therefore, it is enough to prove that
$$
\ba{c}
\nu g_{p, p+\nu}(\tau_1) - (p+\nu) g_{p, p-2+\nu}(\tau_2) \geq \tau_1 g_{p, p+\nu}'(\tau_2),
\ea
$$
or, equivalently,
\begin{equation}\label{eq:monotone_fraction_proof1}
\ba{c}
\frac{ \nu g_{p, p+\nu}(\tau_1) - (p+\nu) g_{p, p-2+\nu}(\tau_2) }{ \tau_1 }
\geq g_{p, p+\nu}'(\tau_2).
\ea
\end{equation}
But this immediately follows from Lemma~\ref{lemma:auxiliary_function_for_monotone_fraction}
using \eqref{eq:proposition:monotone_fraction:proof:eq1}.\qed
\end{proof}

It remains to prove Proposition~\ref{proposition:monotone_fraction1}. For this, we need
one more lemma.

\begin{lemma}\label{lemma:auxiliary_inequality_for_monotone_fraction1}
For any integer $p \geq 0$, and any real $\nu, \tau \in [0, 1]$, we have
\begin{equation}\label{eq:auxiliary_inequality_for_monotone_fraction1}
\ba{c}
(p+\nu) g_{p, p-2+\nu}(\tau) \geq -(1 - (1-\tau)^{1-\nu}) g_{p, p+\nu}'(\tau).
\ea
\end{equation}
\end{lemma}

\begin{proof}
As usual, we use induction in $p$. The base case $p=0$ is trivial, since $g_{p, p-2+\nu}(\tau) = 1$,
while $g_{p, p+\nu}'(\tau) = 0$ (see Definition~\ref{definition:polynomials}).
To prove the general case $p \geq 1$, we assume that
\eqref{eq:auxiliary_inequality_for_monotone_fraction1} is already true
for all integer $0 \leq p' \leq p-1$.

Our first step is to show that
\begin{equation}\label{eq:auxiliary_inequality_for_monotone_fraction1_proof1}
\ba{rl}
(p+\nu) g_{p, p-2+\nu}(\tau) \geq& -(1-\tau^2) (1 - (1-\tau)^{1-\nu}) g_{p-1, p+\nu}''(\tau) \\
&- (p+\nu) (1-\nu) \tau g_{p-1, p-2+\nu}(\tau).
\ea
\end{equation}
If $p=1$, we have $g_{p, p-2+\nu}(\tau) = -(1-\nu) \tau$,
while $g_{p-1, p+\nu}''(\tau) = 0$ and
$g_{p-1, p-2+\nu}(\tau) = 1$ (see Definition~\ref{definition:polynomials}),
so \eqref{eq:auxiliary_inequality_for_monotone_fraction1_proof1} is indeed true.
To justify it for all other $p \geq 2$, we proceed as follows.
By Lemma~\ref{lemma:formula_for_polynomials_without_derivatives}, we know that
$$
\ba{rl}
g_{p, p-2+\nu}(\tau) =&
(1-\tau^2) (p-1) (p-2+\nu) g_{p-2, p-4+\nu}(\tau) \\
& - (1-\nu) \tau g_{p-1, p-2+\nu}(\tau).
\ea
$$
Therefore, \eqref{eq:auxiliary_inequality_for_monotone_fraction1_proof1} is equivalent to
$$
\ba{c}
(p+\nu) (p-1) (p-2+\nu) g_{p-2, p-4+\nu}(\tau)
\geq -(1 - (1-\tau)^{1-\nu}) g_{p-1, p+\nu}''(\tau).
\ea
$$
By our inductive assumption \eqref{eq:auxiliary_inequality_for_monotone_fraction1},
applied to $p' := p-2$, we already have
$$
\ba{c}
(p-2+\nu) g_{p-2, p-4+\nu}(\tau) \geq -(1 - (1-\tau)^{1-\nu}) g_{p-2, p-2+\nu}'(\tau).
\ea
$$
At the same time, by Lemma~\ref{lemma:formula_for_derivative_of_polynomial_3},
$$
\ba{c}
(p+\nu) (p-1) g_{p-2, p-2+\nu}'(\tau) = g_{p-1, p+\nu}''(\tau).
\ea
$$
Thus, \eqref{eq:auxiliary_inequality_for_monotone_fraction1_proof1} is established.

Now we estimate the right-hand side
in \eqref{eq:auxiliary_inequality_for_monotone_fraction1_proof1}.
Applying Lemma~\ref{lemma:formula_for_derivative_of_polynomial_2} and the fact that
$g_{p-1, p+\nu}'(\tau) \geq 0$ (Proposition~\ref{proposition:monotonicity_of_polynomials}),
we obtain
$$
\ba{rl}
g_{p, p+\nu}'(\tau) &= (1-\tau^2) g_{p-1, p+\nu}''(\tau)
+ \nu \tau g_{p-1, p+\nu}'(\tau) + (p+\nu) g_{p-1, p-2+\nu}(\tau) \\
&\geq (1-\tau^2) g_{p-1, p+\nu}''(\tau) + (p+\nu) g_{p-1, p-2+\nu}(\tau).
\ea
$$
From this, it follows that
$$
\ba{c}
(1-\tau^2) g_{p-1, p+\nu}''(\tau) \leq g_{p, p+\nu}'(\tau) - (p+\nu) g_{p-1, p-2+\nu}(\tau).
\ea
$$

Substituting the above equation
into \eqref{eq:auxiliary_inequality_for_monotone_fraction1_proof1}, we obtain
$$
\ba{rl}
(p+\nu) g_{p, p-2+\nu}(\tau) &\geq -(1 - (1-\tau)^{1-\nu}) g_{p, p+\nu}'(\tau) \\
& + (p+\nu) (1 - (1-\nu) \tau - (1-\tau)^{1-\nu}) g_{p-1, p-2+\nu}(\tau).
\ea
$$
\balert{Since $g_{p-1, p-2+\nu}(\tau) \geq 0$ (by
Proposition~\ref{proposition:nonnegativity_of_polynomials}),
it only remains to show that $(1-\tau)^{1-\nu} \leq 1 - 
(1-\nu) \tau$.} But this follows from the concavity of power
function $\tau \mapsto (1 - \tau)^{1-\nu}$.\qed
\end{proof}

Now we can give the proof of Proposition~\ref{proposition:monotone_fraction1}:

\begin{proof}
Since \eqref{eq:proposition:monotone_fraction1} is differentiable,
it suffices to prove that its derivative
\alert{$$
\ba{c}
\frac{ ( 1 - (1-\tau)^\nu ) g_{p, p+\nu}'(\tau)
- \nu (1-\tau)^{\nu-1} g_{p, p+\nu}(\tau) }{ (1 - (1-\tau)^\nu)^2 }
\ea
$$}
is \alert{non-positive} for all $0 < \tau < 1$. By Lemma~\ref{lemma:formula_for_qmpg}, we have
$$
\ba{c}
\nu g_{p, p+\nu}(\tau) = \tau g_{p, p+\nu}'(\tau) + (p+\nu) g_{p, p-2+\nu}(\tau).
\ea
$$
Thus, we need to show that
\alert{$$
(1-\tau)^{\nu-1} \left( \tau g_{p, p+\nu}'(\tau) + (p+\nu)
g_{p, p-2+\nu} (\tau) \right) \geq (1 - (1-\tau)^\nu) g_
{p, p+\nu}'(\tau),
$$
or, equivalently (by multiplying both sides by $(1-\tau)^
{1-\nu}$), that
$$
\ba{c}
\tau g_{p, p+\nu}'(\tau) + (p+\nu) g_{p, p-2+\nu}
(\tau) \geq ((1-\tau)^{1-\nu} - 1+\tau) g_{p, p+\nu}'(\tau),
\ea
$$
or, equivalently (by moving the first term into the
right-hand side), that}
$$
\ba{c}
(p+\nu) g_{p, p-2+\nu}(\tau) \geq -(1 - (1-\tau)^{1-\nu}) g_{p, p+\nu}'(\tau).
\ea
$$
But this is given by Lemma~\ref{lemma:auxiliary_inequality_for_monotone_fraction1}.\qed
\end{proof}

To conclude this section, let us discuss the optimality
of Theorem~\ref{theorem:holder_constant_of_polynomials}.

\balert{For odd values of $p$, the obtained constant $\prod_
{i=1}^p (\nu+i)$ turns out to be optimal.} Indeed, using
$\tau_1 := 0$, $\tau_2 := 1$ in \eqref{eq:def_H_p_nu} and
taking into account
Proposition~\ref{proposition:boundary_values_of_polynomials},
we obtain that
$$
\ba{c}
H_{p, \nu} \geq \frac{g_{p, p+\nu}(\tau_2) - g_{p, p+\nu}(\tau_1)}{(\tau_2 - \tau_1)^\nu}
= g_{p, p+\nu}(1) = \prod_{i=1}^p (\nu+i).
\ea
$$

However, for even $p$, this constant is suboptimal. For
example, consider the case when $p=2$. We know that
$$
\ba{c}
g_{2, 2+\nu}(\tau) = (\nu+2) (\nu \tau^2 + 1).
\ea
$$
The corresponding optimal constant, according to
Proposition~\ref{proposition:monotone_fraction}, is
$$
\ba{c}
H_{2, \nu} = \max\limits_{0 \leq \tau < 1}
\frac{ g_{2, 2+\nu}(1) - g_{2, 2+\nu}(\tau) }{ (1-\tau)^\nu }
= \nu (\nu+2) \max\limits_{0 \leq \tau < 1} (1-\tau)^{1-\nu}
(1+\tau).
\ea
$$
Note that this maximization problem is logarithmically concave in $\tau$.
Taking the logarithm and setting the derivative to zero, we find that the maximal point
corresponds to $\tau := \frac{\nu}{2-\nu} \in [0, 1]$, and the corresponding optimal value is
$$
\ba{c}
H_{2, \nu} = \nu (\nu+2) \frac{2^{2-\nu} (1-\nu)^{1-\nu}}{(2-\nu)^{2-\nu}} \leq (\nu+1) (\nu+2).
\ea
$$
Of course, the last inequality is strict for all $0 \leq \nu < 1$.

\section{H\"older Continuity of Derivatives of Powers of Euclidean Norm}
\label{section:holder_contiuity_of_derivatives}

We have established the main properties of polynomials
$g_{p, q}$ and obtained an explicit upper bound on their
H\"older constant. Hence, we are ready to prove the
H\"older continuity of derivatives of powers of
Euclidean norm. Let us start with a simple result, that
gives us a lower bound on the H\"older constant.

\begin{theorem}\label{theorem:lower_bound_for_holder_constant}
For any integer $p \geq 0$, and any real $\nu \in [0, 1]$, the H\"older constant
of $D^p f_{p+\nu}$, corresponding to degree $\nu$, cannot be smaller than
\begin{equation}\label{eq:definition_of_C_p_nu}
\ba{c}
C_{p, \nu} :=
\begin{cases}
\prod_{i=1}^p (\nu+i), & \text{if $p$ is even}, \\
2^{1-\nu} \prod_{i=1}^p (\nu+i), & \text{if $p$ is odd}.
\end{cases}
\ea
\end{equation}
\end{theorem}

\begin{proof}
According to \eqref{eq:norm_of_symmetric_multilinear_operator}, we need to show that
$$
\ba{c}
| D^p f_{p+\nu}(x_2)[h]^p - D^p f_{p+\nu}(x_1)[h]^p |
\geq C_{p, \nu} \| x_2 - x_1 \|^\nu
\ea
$$
for some $x_1, x_2 \in \E$ and some unit $h \in \E$.
\balert{Let us choose an arbitrary unit vector $h \in \E$,
and set $x_2 := h$. By
Theorem~\ref{theorem:formula_for_derivative} and
Proposition~\ref{proposition:boundary_values_of_polynomials},}
$$
\ba{c}
D^p f_{p+\nu}(x_2)[h]^p = \| x_2 \|^\nu g_{p, p+\nu}(1)
= \prod_{i=1}^p (\nu+i).
\ea
$$
To specify $x_1$, we consider two cases.

\balert{If $p$ is even, set $x_1 := 0$. Then, $D^p
f_{p+\nu}(x_1)[h]^p = 0$ by
Theorem~\ref{theorem:formula_for_derivative}, and}
$$
\ba{c}
| D^p f_{p+\nu}(x_2)[h]^p - D^p f_{p+\nu}(x_1)[h]^p |
= \prod_{i=1}^p (\nu+i),
\ea
$$
which is exactly $C_{p, \nu} \| x_2 - x_1 \|^\nu$. 
\balert{If $p$ is odd, we take $x_1 := -h$. This gives us}
$$
\ba{c}
D^p f_{p+\nu}(x_1)[h]^p = \| x_1 \|^\nu g_{p, p+\nu}(-1)
= -\prod_{i=1}^p (\nu+i),
\ea
$$
\balert{where we apply Proposition~\ref{proposition:parity}
to rewrite $g_{p, p+\nu}(-1) = g_{p, p+\nu}(1)$. Hence,}
$$
\ba{c}
| D^p f_{p+\nu}(x_2)[h]^p - D^p f_{p+\nu}(x_1)[h]^p |
= 2 \prod_{i=1}^p (\nu+i),
\ea
$$
which is again precisely $C_{p, \nu} \| x_2 - x_1 \|^\nu$.\qed
\end{proof}

Next we prove H\"older continuity with the optimal
constant along any line, passing through the origin.

\begin{theorem}\label{theorem:holder_continuity_on_line}
For any integer $p \geq 0$, and any real $\nu \in [0, 1]$,
\balert{the restriction of $D^p f_{p+\nu}$ to a line, passing
through the origin, is $\nu$-H\"older continuous with
constant $C_{p, \nu}$.}
\end{theorem}

\begin{proof}
Let $x_1, x_2 \in \E$ be arbitrary points, lying on a line,
passing through the origin, and let $h \in \E$ be an
arbitrary unit vector. According to
\eqref{eq:norm_of_symmetric_multilinear_operator} and
Theorem~\ref{theorem:formula_for_derivative}, we need to
show that
$$
\ba{c}
| \| x_2 \|^\nu g_{p, p+\nu}(\tau_2) - \| x_1 \|^\nu g_{p, p+\nu}(\tau_1) |
\leq C_{p, \nu} \| x_2 - x_1 \|^\nu,
\ea
$$
where $\tau_1 := \tau_h(x_1)$, $\tau_2 := \tau_h(x_2)$.

Observe that this inequality is symmetric in $x_1$ and $x_2$ and is invariant
when we replace the pair $(x_1, x_2)$ with $(-x_1, -x_2)$.
Therefore, we can assume that $\| x_2 \| \geq \| x_1 \|$ and $\tau_2 \geq 0$.

\balert{Since $x_1$ and $x_2$ lie on a line, passing through
the origin, $\tau_1$ and $\tau_2$ can differ only in sign.
Hence, by Proposition~\ref{proposition:parity}, we have
two options: either $g_{p, p+\nu}(\tau_1) = g_{p,
p+\nu}(\tau_2)$ or $g_{p, p+\nu}(\tau_1) = -g_{p,
p+\nu}(\tau_2)$.

\underline{Case I}. Suppose $g_{p, p+\nu}(\tau_1) = g_{p,
p+\nu}(\tau_2)$ (while $\tau_1$ can be of any sign). Then,}
$$
\ba{c}
| \| x_2 \|^\nu g_{p, p+\nu}(\tau_2) - \| x_1 \|^\nu g_{p, p+\nu}(\tau_1) |
= |g_{p, p+\nu}(\tau_2)| (\| x_2 \|^\nu - \| x_1 \|^\nu).
\ea
$$
By Proposition~\ref{proposition:maximal_absolute_value_of_polynomial} and
\eqref{eq:definition_of_C_p_nu}, we know that
\begin{equation}\label{eq:holder_continuity_on_line_proof1}
\ba{c}
|g_{p, p+\nu}(\tau_2)| \leq \prod_{i=1}^p (\nu+i) \leq C_{p, \nu}.
\ea
\end{equation}
\balert{Thus, it suffices to prove that $\| x_2 \|^\nu - \|
x_1 \|^\nu \leq \| x_2 - x_1 \|^\nu$.} But this follows from
the well-known inequality $r_2^\nu - r_1^\nu \leq (r_2 -
r_1)^\nu$ (which is valid for any real $0 \leq r_1 \leq
r_2$) combined with the reverse triangle inequality.

\balert{\underline{Case II}. \alert{Suppose $g_{p, p+\nu}
(\tau_1) = -g_{p, p+\nu}(\tau_2)$ ($\neq 0$)}. By
Proposition~\ref{proposition:parity} \alert{and
Proposition~\ref{proposition:nonnegativity_of_polynomials}},
this happens only if $p$ is odd and $\tau_1 \leq 0$. Thus
$\tau_1 = -\tau_2$, and}
\begin{equation}\label{eq:holder_continuity_on_line_proof4}
\ba{c}
| \| x_2 \|^\nu g_{p, p+\nu}(\tau_2) - \| x_1 \|^\nu g_{p, p+\nu}(\tau_1) |
= |g_{p, p+\nu}(\tau_2)| (\| x_2 \|^\nu + \| x_1 \|^\nu).
\ea
\end{equation}
\balert{Due to \eqref{eq:holder_continuity_on_line_proof1},
it remains to prove that $\| x_2 \|^\nu + \| x_1 \|^\nu \leq
2^{1-\nu} \| x_2 - x_1
\|^\nu$. But this is immediate. Indeed, $\| x_2 \|^\nu + \|
x_1 \|^\nu \leq 2^{1-\nu} (\| x_2 \| + \| x_1 \|)^\nu$ by
the concavity of the
power function $t \mapsto t^\nu$, while $\| x_2 \| + \| x_1
\| = \| x_2 - x_1 \|$ since the segment $
[x_1, x_2]$ contains the origin.}.\qed
\end{proof}

Our final step is to extend H\"older continuity from lines,
passing through the origin, onto the whole space. The main
instrument for doing this is exploiting H\"older
continuity of $g_{p, p+\nu}$, that we studied in
Section~\ref{section:holder_constants_of_polynomials}.

\begin{theorem}\label{theorem:holder_continuity_on_whole_space}
For any integer $p \geq 0$, and any real $\nu \in [0, 1]$,
$D^p f_{p+\nu}$ is $\nu$-H\"older continuous
with constant
\begin{equation}\label{eq:definition_tilde_A}
\ba{c}
A_{p, \nu} :=
\begin{cases}
(p-1)!! \prod_{i=1}^{p/2} (\nu + 2i) + H_{p, \nu}, & \text{if $p$ is even}, \\
2^{1-\nu} \prod_{i=1}^p (\nu + i), & \text{if $p$ is odd},
\end{cases}
\ea
\end{equation}
where $H_{p, \nu}$ is the constant of $\nu$-H\"older continuity
of $g_{p, p+\nu}$ on $[0, 1]$. In particular, $D^p f_{p+\nu}$ is $\nu$-H\"older continuous
with constant
$$
\ba{c}
\tilde{A}_{p, \nu} :=
\begin{cases}
(p-1)!! \prod_{i=1}^{p/2} (\nu + 2i) + \prod_{i=1}^p (\nu+i), & \text{if $p$ is even}, \\
2^{1-\nu} \prod_{i=1}^p (\nu+i), & \text{if $p$ is odd}.
\end{cases}
\ea
$$
\end{theorem}

\begin{proof}
First of all, observe that the constant $A_{p, \nu}$ is
not smaller than the corresponding lower bound $C_{p,
\nu}$ given by
Theorem~\ref{theorem:lower_bound_for_holder_constant}:
\begin{equation}\label{eq:holder_continuity_on_whole_space_proof_lower_bound}
\ba{c}
C_{p, \nu} \leq A_{p, \nu}.
\ea
\end{equation}
Indeed, for odd values of $p$, these constants coincide. When $p$ is even,
\eqref{eq:holder_continuity_on_whole_space_proof_lower_bound} follows from
the following trivial lower bound for the H\"older constant $H_{p, \nu}$:
$$
\ba{c}
H_{p, \nu} \geq g_{p, p+\nu}(1) - g_{p, p+\nu}(0)
= \prod_{i=1}^p (\nu+i) - (p-1)!! \prod_{i=1}^{p/2} (\nu+2i),
\ea
$$
where the last equality is due to Proposition~\ref{proposition:boundary_values_of_polynomials}.

Secondly, observe that we only need to prove the first
claim, since the other one follows directly from the first
one and
Theorem~\ref{theorem:holder_constant_of_polynomials}.

Let $x_1, x_2 \in \E$, and let $h \in \E$ be an arbitrary
unit vector. In view of
\eqref{eq:norm_of_symmetric_multilinear_operator} and
Theorem~\ref{theorem:formula_for_derivative}, we need to
show that
\begin{equation}\label{eq:holder_continuity_on_whole_space_proof1}
\ba{c}
| \| x_2 \|^\nu g_{p, p+\nu}(\tau_2) - \| x_1 \|^\nu g_{p, p+\nu}(\tau_1) |
\leq A_{p, \nu} \| x_2 - x_1 \|^\nu,
\ea
\end{equation}
where $\tau_1 := \tau_h(x_1)$, $\tau_2 := \tau_h(x_2)$.

Due to invariance of the above inequality to transformations of the form $(x_1, x_2) \mapsto (x_2, x_1)$
and $(x_1, x_2) \mapsto (-x_1, -x_2)$, we can assume in what follows, that $\| x_1 \| \leq \| x_2 \|$
and $\tau_2 \geq 0$. Furthermore, we can also assume that $x_1 \neq 0$ (and hence $x_2 \neq 0$),
since otherwise the claim follows from Theorem~\ref{theorem:holder_continuity_on_line}.

There are now several cases to consider.

\balert{\underline{Case I}. \alert{Suppose $g_{p, p+\nu} 
(\tau_1) < 0$.} By Propositions \ref{proposition:parity} and
\ref{proposition:nonnegativity_of_polynomials}, this
happens only if $p$ is odd and $\tau_1 \leq 0$. Then,
$g_{p, p+\nu}(\tau_1) = -g_{p, p+\nu}(-\tau_1)$, and}
$$
\ba{c}
| \| x_2 \|^\nu g_{p, p+\nu}(\tau_2) - \| x_1 \|^\nu g_{p, p+\nu}(\tau_1) |
= \| x_2 \|^\nu g_{p, p+\nu}(\tau_2) + \| x_1 \|^\nu g_{p, p+\nu}(-\tau_1),
\ea
$$
where we have removed the absolute value sign, because all terms in the right-hand side
are non-negative (see Proposition~\ref{proposition:nonnegativity_of_polynomials}).

Since $p$ is odd, $g_{p, p+\nu}(0) = 0$ (see Proposition~\ref{proposition:parity}).
Therefore, by the definition of $H_{p, \nu}$, it follows that
$$
\ba{rl}
g_{p, p+\nu}(-\tau_1) &= g_{p, p+\nu}(-\tau_1) - g_{p, p+\nu}
(0) \leq H_{p, \nu} (-\tau_1^\nu), \\
g_{p, p+\nu}(\tau_2) &= g_{p, p+\nu}(\tau_2) - g_{p, p+\nu}(0) \leq H_{p, \nu} \tau_2^\nu.
\ea
$$
Combining this with the concavity of power function $t \mapsto t^\nu$,
we obtain
$$
\ba{rl}
\| x_2 \|^\nu g_{p, p+\nu}(\tau_2) + \| x_1 \|^\nu g_{p, p+\nu}(-\tau_1)
&\leq H_{p, \nu} ( (\| x_2 \| \tau_2)^\nu + (-\| x_1 \| \tau_1)^\nu ) \\
&\leq 2^{1-\nu} H_{p, \nu} ( \| x_2 \| \tau_2 - \| x_1 \| \tau_1 )^\nu.
\ea
$$
Note that $2^{1-\nu} H_{p, \nu} \leq A_{p, \nu}$ by
Theorem~\ref{theorem:holder_constant_of_polynomials}.
\balert{Thus, it remains to show that $\| x_2 \| \tau_2 - \|
x_1 \| \tau_1 \leq \| x_2 - x_1 \|$. But this follows from
the Cauchy-Schwartz inequality since $\| x_2
\| \tau_2 - \| x_1 \| \tau_1 = \langle B (x_2 - x_1), h
\rangle$ by the definition of $\tau_1$, $\tau_2$.}

\underline{Case II}. Now suppose $g_{p, p+\nu}(\tau_1)
\geq 0$ \alert{(while $\tau_1$ can have any sign)}. We prove
\eqref{eq:holder_continuity_on_whole_space_proof1} by
proving separately two inequalities with the removed
absolute value sign.

First, we show that
\begin{equation}\label{eq:holder_continuity_on_whole_space_proof2}
\ba{c}
\| x_1 \|^\nu g_{p, p+\nu}(\tau_1) - \| x_2 \|^\nu g_{p, p+\nu}(\tau_2)
\leq A_{p, \nu} \| x_2 - x_1 \|^\nu.
\ea
\end{equation}

Let $x_2':= \frac{\| x_1 \|}{\| x_2 \|} x_2$ be the
radial projection of $x_2$ onto the sphere with radius $r :=
\| x_1 \|$, centered at the origin. Note that
\begin{equation}\label{eq:holder_continuity_on_whole_space_proof2_1}
\ba{c}
\tau_2' := \tau_h(x_2') = \tau_2,
\qquad \| x_2' \| = r \leq \| x_2 \|,
\qquad \| x_2' - x_1 \| \leq \| x_2 - x_1 \|.
\ea
\end{equation}
\alert{The first two relations are evident. The last one
follows from the fact that projections onto convex sets
decrease distances, and can be explicitly proved as follows.
First, by the Cauchy-Schwartz inequality, we have $\langle B
x_1, x_2 \rangle \leq \rho \| x_2 \|^2$, where $\rho := 
\frac{\| x_1 \|}{\| x_2 \|} \leq 1$. Therefore, using the
fact that $x_2' = \rho x_2$, we obtain
$$
\begin{aligned}
\| x_2 & - x_1 \|^2 - \| x_2' - x_1 \|^2
= \| x_2 \|^2 - \| x_2' \|^2 - 2 \langle B x_1, x_2 - x_2'
\rangle \\
&= (1 - \rho^2) \| x_2 \|^2 - 2 (1 - \rho) \langle
B x_1, x_2 \rangle \geq (1 - \rho^2) \| x_2 \|^2 - 2 (1 -
\rho) \rho \| x_2 \|^2 \\
&= (1 - \rho)^2 \| x_2 \|^2 \geq 0.
\end{aligned}
$$}

Since $g_{p, p+\nu}(\tau_2) \geq 0$ (Proposition~\ref{proposition:nonnegativity_of_polynomials}), from 
\eqref{eq:holder_continuity_on_whole_space_proof2_1} it
follows that
$$
\ba{c}
\| x_1 \|^\nu g_{p, p+\nu}(\tau_1) - \| x_2 \|^\nu g_{p, p+\nu}(\tau_2)
\leq r^\nu (g_{p, p+\nu}(\tau_1) - g_{p, p+\nu}(\tau_2')).
\ea
$$
At the same time, by Theorem~\ref{theorem:extending_holder_continuity_for_polynomials},
$$
\ba{c}
g_{p, p+\nu}(\tau_1) - g_{p, p+\nu}(\tau_2') \leq \tilde{H}_{p, \nu} |\tau_2' - \tau_1|^\nu.
\ea
$$
Hence,
\begin{equation}\label{eq:holder_continuity_on_whole_space_proof2_2}
\ba{c}
\| x_1 \|^\nu g_{p, p+\nu}(\tau_1) - \| x_2 \|^\nu g_{p, p+\nu}(\tau_2)
\leq \tilde{H}_{p, \nu} (r |\tau_2' - \tau_1|)^\nu.
\ea
\end{equation}
Note that
$$
\ba{c}
r |\tau_2' - \tau_1| = | \| x_2' \| \tau_2' - \| x_1 \| \tau_1 |
= | \langle B (x_2' - x_1), h \rangle |.
\ea
$$
Therefore, by Cauchy-Schwartz inequality and
\eqref{eq:holder_continuity_on_whole_space_proof2_1}, we have
$$
\ba{c}
r |\tau_2' - \tau_1| \leq \| x_2' - x_1 \| \leq \| x_2 - x_1 \|.
\ea
$$
Substituting this into \eqref{eq:holder_continuity_on_whole_space_proof2_2}, we obtain
$$
\ba{c}
\| x_1 \|^\nu g_{p, p+\nu}(\tau_1) - \| x_2 \|^\nu g_{p, p+\nu}(\tau_2)
\leq \tilde{H}_{p, \nu} \| x_2 - x_1 \|^\nu.
\ea
$$
This finishes the proof of \eqref{eq:holder_continuity_on_whole_space_proof2}, because
$\tilde{H}_{p, \nu} \leq A_{p, \nu}$
by Theorem~\ref{theorem:holder_constant_of_polynomials}.

It remains to show the reverse inequality
\begin{equation}\label{eq:holder_continuity_on_whole_space_proof3}
\ba{c}
\| x_2 \|^\nu g_{p, p+\nu}(\tau_2) - \| x_1 \|^\nu g_{p, p+\nu}(\tau_1)
\leq A_{p, \nu} \| x_2 - x_1 \|^\nu.
\ea
\end{equation}
For this, we consider two subcases.

\balert{\underline{Case II(a)}. Suppose $\tau_1 \geq \tau_2$.
Let $x_1' := \frac{\langle B x_1, x_2 \rangle}{ \| x_2 \|^2
} x_2$ be the projection of $x_1$ onto the line, connecting
$x_2$ with the origin, and let $\tau_1':= \tau_h
(x_1')$. Then,}
\begin{equation}\label{eq:holder_continuity_on_whole_space_proof2a_1}
\ba{c}
\| x_1' \| \leq \| x_1 \|, \qquad \| x_2 - x_1' \| \leq \| x_2 - x_1 \|.
\ea
\end{equation}
Furthermore,
\begin{equation}\label{eq:holder_continuity_on_whole_space_proof2a_2}
\ba{c}
g_{p, p+\nu}(\tau_1') \leq g_{p, p+\nu}(\tau_2).
\ea
\end{equation}
Indeed, if $\langle B x_1, x_2 \rangle \geq 0$, then $\tau_1' = \tau_2$
and $g_{p, p+\nu}(\tau_1') = g_{p, p+\nu}(\tau_2)$; otherwise $\tau_1' = -\tau_2$,
and hence $g_{p, p+\nu}(\tau_1') = (-1)^p g_{p, p+\nu}(\tau_2)$
(Proposition~\ref{proposition:parity}), which either coincides with $g_{p, p+\nu}(\tau_2)$
when $p$ is even, or becomes $-g_{p, p+\nu}(\tau_2) \leq 0$ when $p$ is odd
(see Proposition~\ref{proposition:nonnegativity_of_polynomials}).

Since $g_{p, p+\nu}(\tau_2) \leq g_{p, p+\nu}(\tau_1)$
(Proposition~\ref{proposition:monotonicity_of_polynomials}),
\balert{it follows from
\eqref{eq:holder_continuity_on_whole_space_proof2a_2} that
$g_{p, p+\nu}(\tau_1') \leq g_{p, p+\nu}(\tau_1)$.}
\alert{Using also
\eqref{eq:holder_continuity_on_whole_space_proof2a_1} and
the fact that $g_{p, p+\nu}(\tau_1) \geq 0$, we obtain $\|
x_1' \|^{\nu} g_{p, p+\nu}(\tau_1') \leq \| x_1' \|^{\nu}
g_{p, p+\nu}(\tau_1) \leq \| x_1 \|^{\nu} g_{p, p+\nu}
(\tau_1)$. Thus,}
$$
\ba{c}
\| x_2 \|^\nu g_{p, p+\nu}(\tau_2) - \| x_1 \|^\nu g_{p, p+\nu}(\tau_1)
\leq \| x_2 \|^\nu g_{p, p+\nu}(\tau_2) - \| x_1' \|^\nu g_{p, p+\nu}(\tau_1').
\ea
$$
Note that in the right-hand side, we have the difference of the derivatives
$D^p f_{p+\nu}(x_2)[h]^p$ and $D^p f_{p+\nu}(x_1')[h]^p$,
where the points $x_1'$ and $x_2$ lie on a line, passing through the origin.
Therefore, from Theorem~\ref{theorem:holder_continuity_on_line}, it follows that
$$
\ba{c}
\| x_2 \|^\nu g_{p, p+\nu}(\tau_2) - \| x_1 \|^\nu g_{p, p+\nu}(\tau_1)
\leq C_{p, \nu} \| x_2 - x_1' \|^\nu,
\ea
$$
which proves \eqref{eq:holder_continuity_on_whole_space_proof3},
in view of \eqref{eq:holder_continuity_on_whole_space_proof_lower_bound}
and \eqref{eq:holder_continuity_on_whole_space_proof2a_1}.

\underline{Case II(b)}. Now suppose $\tau_1 \leq \tau_2$. 
\alert{Denote by $\tilde{H}_{p, \nu}$ the constant of
$\nu$-H\"older continuity of the polynomial $g_{p, p+\nu}$
on the interval $[-1, 1]$.} To prove 
\eqref{eq:holder_continuity_on_whole_space_proof3}, it suffices to show that
\begin{equation}\label{eq:holder_continuity_on_whole_space_proof2b}
\ba{rl}
\| x_2 \|^\nu g_{p, p+\nu}&(\tau_2) - \| x_1 \|^\nu g_{p, p+\nu}(\tau_1) \\
&\leq g_{p, p+\nu}(0) (\| x_2 \|^\nu - \| x_1 \|^\nu)
+ \alert{\tilde{H}_{p, \nu}} ( \| x_2 \| \tau_2 - \| x_1 \|
\tau_1
)^\nu,
\ea
\end{equation}
\balert{Indeed, recall that $\| x_2 \|^\nu - \| x_1 \|^\nu
\leq \| x_2 - x_1 \|^\nu$. Also
$$
\ba{c}
( \| x_2 \| \tau_2 - \| x_1 \| \tau_1 )^\nu
= \langle B (x_2 - x_1), h \rangle^\nu \leq \| x_2 - x_1 \|^\nu
\ea
$$
by the Cauchy-Schwartz inequality.} Therefore, if 
\eqref{eq:holder_continuity_on_whole_space_proof2b} is true, then
$$
\ba{c}
\| x_2 \|^\nu g_{p, p+\nu}(\tau_2) - \| x_1 \|^\nu g_{p, p+\nu}(\tau_1)
\leq (g_{p, p+\nu}(0) + \alert{\tilde{H}_{p, \nu}}) \| x_2 -
x_1 \|^\nu,
\ea
$$
\alert{where $g_{p, p+\nu}(0) + \tilde{H}_{p, \nu} \leq A_
{p, \nu}$ in view of Proposition~\ref{proposition:boundary_values_of_polynomials},
Theorem~\ref{theorem:extending_holder_continuity_for_polynomials}
and Theorem~\ref{theorem:holder_constant_of_polynomials}.}
Thus, it remains to show \eqref{eq:holder_continuity_on_whole_space_proof2b},
or, equivalently, that
$$
\ba{rl}
\| x_2 \|^\nu &(g_{p, p+\nu}(\tau_2) - g_{p, p+\nu}(0)) \\
&\leq \| x_1 \|^\nu (g_{p, p+\nu}(\tau_1) - g_{p, p+\nu}(0))
+ \alert{\tilde{H}_{p, \nu}} ( \| x_2 \| \tau_2 - \| x_1 \|
\tau_1 )^\nu.
\ea
$$
Denote $\rho := \frac{\| x_1 \|}{\| x_2 \|} \in [0, 1]$. We need to prove that
\begin{equation}\label{eq:holder_continuity_on_whole_space_proof2b_1}
\ba{c}
g_{p, p+\nu}(\tau_2) - g_{p, p+\nu}(0)
\leq \rho^\nu (g_{p, p+\nu}(\tau_1) - g_{p, p+\nu}(0))
+ \alert{\tilde{H}_{p, \nu}} ( \tau_2 - \rho \tau_1 )^\nu.
\ea
\end{equation}
Note that the right-hand side of this inequality, as a function of $\rho \in [0, 1]$,
is \emph{concave} (and well-defined, since $\tau_1 \leq \tau_2$).
Hence, to justify \eqref{eq:holder_continuity_on_whole_space_proof2b_1},
we only need to prove the following two boundary cases:
$$
\ba{rl}
\rho=0: &\qquad g_{p, p+\nu}(\tau_2) - g_{p, p+\nu}(0) \leq
\alert{\tilde{H}_{p, \nu}} \tau_2^\nu. \\
\rho=1: &\qquad g_{p, p+\nu}(\tau_2) - g_{p, p+\nu}(\tau_1)
\leq \alert{\tilde{H}_{p, \nu}} (\tau_2-\tau_1)^\nu. \\
\ea
$$
But both of them follow from the definition of $\alert{\tilde{H}_{p, \nu}}$.\qed
\end{proof}

\balert{Comparing the result of
Theorem~\ref{theorem:holder_continuity_on_whole_space}
with the lower bound $C_{p, \nu}$, given by
Theorem~\ref{theorem:lower_bound_for_holder_constant}, we
see, that for \emph{odd} values of $p$, the constant}
$\tilde{A}_{p, \nu}$ is \emph{optimal}. Unfortunately, this
is no longer true for \emph{even}
values of $p$. Nevertheless, the constant $\tilde{A}_{p,
\nu}$ is still quite accurate. Indeed, since
$$
\ba{c}
(p-1)!! = \prod_{i=1}^{p/2} (2i-1) \leq \prod_{i=1}^{p/2} (\nu+2i-1),
\ea
$$
we have
$$
\ba{c}
(p-1)!! \prod_{i=1}^{p/2} (\nu+2i) \leq \prod_{i=1}^p (\nu+i).
\ea
$$
\balert{Thus, the constant $\tilde{A}_{p, \nu}$ is at most
two times suboptimal: $\tilde{A}_{p, \nu} \leq 2 C_{p,
\nu}$.}

One may think that the reason, why we obtained a suboptimal
bound for even values of $p$, is related to the fact that
we had used a suboptimal value for the H\"older constant
$H_{p, \nu}$ of the polynomial $g_{p, p+\nu}$ (see the
corresponding discussion at the end of
Section~\ref{section:holder_constants_of_polynomials}).
However, this is not the actual reason. Indeed, let us
look what happens when we use the optimal value for $H_{p,
\nu}$ in the particular case $p=2$. Recall that the
optimal constant in this case is
$$
\ba{c}
H_{2, \nu} = \nu (\nu+2) \frac{2^{2-\nu} (1-\nu)^{1-\nu}}{(2-\nu)^{2-\nu}}.
\ea
$$
Substituting this expression into \eqref{eq:definition_tilde_A}, we obtain an improved estimate
$$
\ba{c}
A_{2, \nu} = \nu+2 + H_{2, \nu}
= (\nu+2) \left( \nu+1 + \nu \frac{2^{2-\nu} (1-\nu)^{1-\nu}}{(2-\nu)^{2-\nu}} \right).
\ea
$$
However, this new estimate is still different from the lower bound
$$
\ba{c}
C_{2, \nu} = (\nu+1) (\nu+2).
\ea
$$
At the same time, for small values of $\nu$, the difference between
$A_{p, \nu}$ and $C_{p, \nu}$ is almost negligible.

\section{Lipschitz Constants of Derivatives of Powers of Euclidean Norm}
\label{section:lipschitz_constants_of_derivatives}

For even values of $p$, our estimate $A_{p, \nu}$ of the H\"older constant
of $D^p f_{p+\nu}$ was suboptimal. It turns out, that in the special case
when $\nu=1$, it is actually very simple to eliminate this drawback
and obtain an optimal constant for \emph{all} values of $p$.
This case corresponds to \emph{Lipschitz} continuity.

\begin{theorem}\label{theorem:lipschitz_constant}
For any integer $p \geq 0$, the derivative $D^p f_{p+1}$ is Lipschitz continuous
with constant
$$
\ba{c}
C_{p, 1} = (p+1)!,
\ea
$$
where $n!$ for a non-negative integer $n$ denotes the factorial of $n$.
\end{theorem}

\begin{proof}\balert{
It suffices to prove that $| D^{p+1} f_{p+1}(x)[h]^{p+1} |
\leq (p+1)!$ for all $x \in \E$ and all unit $h \in \E$. By
Theorem~\ref{theorem:formula_for_derivative}, we have
$D^{p+1} f_{p+1}(x)[h]^{p+1} = g_{p+1, p+1}(\tau_h(x))$.
Since $|\tau_h(x)| \leq 1$, we obtain $| D^{p+1}
f_{p+1}(x)[h]^{p+1} | \leq \max_{[-1, 1]} | g_{p+1, p+1} |$.
The claim now follows from
Proposition~\ref{proposition:maximal_absolute_value_of_polynomial}.\qed
}\end{proof}

\section{Conclusions}

In this work, we have proved that derivatives of powers of
Euclidean norm are H\"older continuous and have obtained
explicit expressions for the corresponding H\"older
\balert{constants. We have shown that our constants are
optimal for odd derivatives and at most two times suboptimal
for the even ones. In the particular case of integer powers,
when the H\"older condition corresponds to the Lipschitz
condition, we have managed to improve our result and
obtained optimal constants in all cases. We believe that
in general it should be possible to obtain optimal
constants for even derivatives as well. However, this
seems to be a difficult problem.}

\begin{acknowledgements}
\balert{The presented results were supported by ERC Advanced
Grant 788368. We are grateful to the anonymous referees for
their attentive reading and valuable comments.}
\end{acknowledgements}


\end{document}